\documentclass[a4paper]{amsart}
\usepackage{amsthm,amsfonts,amsmath,amssymb}
\usepackage[abs]{overpic}
\usepackage{thm-restate}

\newtheorem{theorem}{Theorem}[section]
\newtheorem{proposition}[theorem]{Proposition}
\newtheorem{corollary}[theorem]{Corollary}
\newtheorem{claim}[theorem]{Claim}

\newtheorem{example}[theorem]{Example}

\begin{document}
\title[DELTA-UNKNOTTING NUMBER FOR TWO-BRIDGE KNOTS]{DELTA-UNKNOTTING NUMBER FOR TWO-BRIDGE KNOTS}
\author{Kazumichi Nakamura}
\email{mi-ka-07130703@ezweb.ne.jp}





\begin{abstract}
The $\Delta$-unknotting number for a knot is defined as the minimum number of $\Delta$-moves needed to deform the knot into the trivial knot. We determine the $\Delta$-unknotting numbers for two-bridge knots of type
$C(2\beta_1, 2\beta_2, ... ,  2\beta_n)$ and type $C(2\beta_1, 2\beta_2, ... , 2\beta_{n-1}, 2\beta_n-1)$, where $\beta_i$ is a positive integer for $1 \leq i \leq n$.
We also discuss two-bridge knots whose $\Delta$-unknotting number is equal to one.
\end{abstract}

\maketitle

\noindent\textbf{keywords:}$\Delta$-move, $\Delta$-unknotting number, $\Delta$-Gordian distance, Conway polynomial, Conway's normal form, two-bridge knot, torus knot, linking number.

\section{Introduction}
In this paper, we study the $\Delta$-unknotting number for two-bridge knots.

In \cite{MaN}, H. Murakami and Y. Nakanishi introduced a local move on regular diagrams of oriented knots and links, called a $\Delta$-move (or $\Delta$-unknotting operation), as illustrated in Figure \ref{fig:delta}.
\begin{figure}[htbp]
 \centering    \includegraphics[width=0.8\linewidth]
    {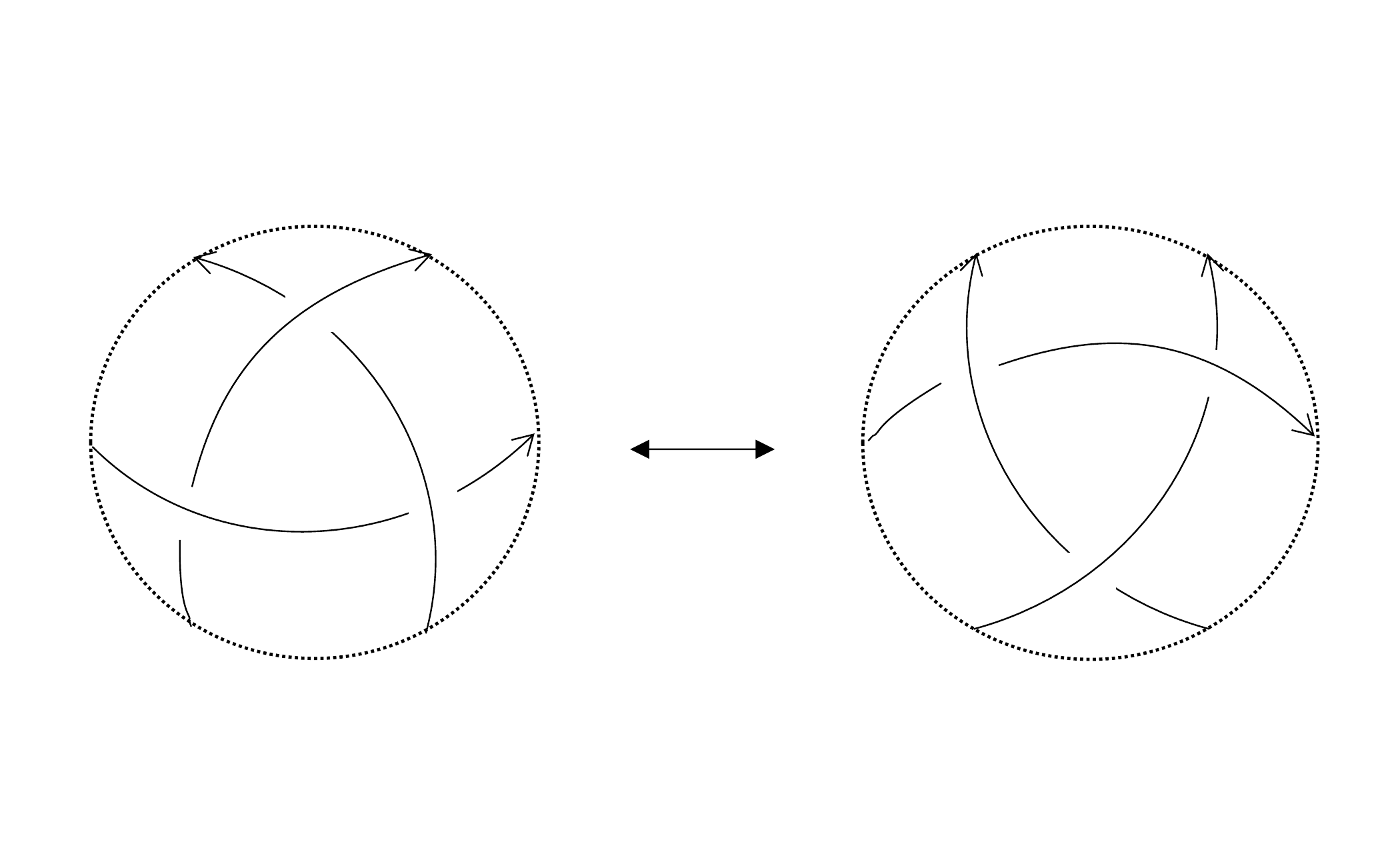}
 \caption{}
 \label{fig:delta}
\end{figure}
They proved that two knots can be deformed into each other by a finite sequence of $\Delta$-moves. 
The $\Delta$-Gordian distance $d_G^{\Delta}(K, K^{'})$ of two oriented knots $K$ and $K^{'}$ is defined as the minimum number of $\Delta$-moves needed to deform a diagram of $K$ into that of $K{'}$. 
The $\Delta$-unknotting number $u^{\Delta}(K)$  of an oriented knot $K$ is defined as the $\Delta$-Gordian distance $d_G^{\Delta}(K, O)$ of $K$ and the trivial knot $O$.

In \cite{Oka1}, M. Okada proved that $u^{\Delta}(K) \geq |a_2(K)|$, where $a_2(K)$ denotes the second coefficient of the Conway polynomial of $K$. 
She also constructed a table of the $\Delta$-unknotting numbers of all prime knots with nine or fewer crossings \cite{Oka1,Oka2}.
Subsequently, in \cite{Naka1,Naka2}, we determined the $\Delta$-unknotting numbers for all prime knots with ten crossings, except for 14 knots. 
While constructing this table, we observed that for many knots (such as torus knots), the $\Delta$-unknotting number is equal to the second coefficient of their Conway polynomials.

In \cite{NaNaU}, we determined the $\Delta$-unknotting numbers for certain classes of oriented knots: torus knots, positive pretzel knots, and positive 3-braids.
For these classes, the $\Delta$-unknotting number coincides with the second coefficient of the Conway polynomial. In particular, the $\Delta$-unknotting number for $T(p,q)$ (the torus knot of type $(p,q)$) is equal to $(p^2-1)(q^2-1)/24$.

However, the $\Delta$-unknotting numbers for the remaining 14 knots have not yet been determined.  
According to \cite{Uchi}, the $\Delta$-unknotting number of $9_{29}$ is now regarded as undetermined due to an error in the diagram illustrating the unknotting procedure in a previous study.
At present, the exact values for some two-bridge knots remain unknown.
Specifically, $u^{\Delta}(10_{14}) = 2$ or $4$, $u^{\Delta}(10_{30}) = 1$ or $3$, $u^{\Delta}(10_{36}) = 1$ or $3$, and $u^{\Delta}(10_{38}) = 1$ or $3$.

\vspace{1em}
As a starting point for resolving these cases, in this paper, we consider fundamental families of two-bridge knots.

In Section \ref{sec3}, we determine the $\Delta$-unknotting numbers for two-bridge knots of type 
$C(2\beta_1, 2\beta_2, ... ,  2\beta_n)$ and type $C(2\beta_1, 2\beta_2, ... , 2\beta_{n-1}, 2\beta_n-1)$, where $\beta_i$ is a positive integer for $1 \leq i \leq n$ (Theorems \ref{thm:even} and \ref{thm:odd}).

\begin{restatable}{theorem}{even} \label{thm:even}
Let $K=C(\alpha_1, \alpha_2, ... , \alpha_n)$ be a two-bridge knot, where $\alpha_i$ is a positive even integer for $1 \leq i \leq n$ and $n$ is even.
Then, we have 
\begin{equation*}
u^{\Delta}(K) = 
\frac{1}{4} \sum_{j=1}^{n/2} \sum_{i=1}^{j} \alpha_{2i-1}\alpha_{2j} \quad
\big(= -a_2(K)\big).
\end{equation*}
\end{restatable}

\begin{restatable}{theorem}{odd} \label{thm:odd}
Let $K=C(\alpha_1, \alpha_2, ... , \alpha_n)$ be a two-bridge knot, where $\alpha_i$ is a positive even integer for $1 \leq i \leq n-1$, and $\alpha_n$ is a positive odd integer. In particular, if $n$ = 1 ($K=C(\alpha_1)$), then $\alpha_1$ is a positive odd integer.
Then, we have $u^{\Delta}(K)=|a_2(K)|$.
\begin{enumerate}
    \item
    If $n$ is even, then
    \[
    u^{\Delta}(K)=
\frac{1}{4} \sum_{j=1}^{n/2} \sum_{i=1}^{j} \alpha_{2i-1}\alpha_{2j} + \frac{1}{8} ( \sum_{i=1}^{n/2} \alpha_{2i-1} )^{2} \quad \big(= a_2(K)\big).
\]
    \item
    If $n$ is odd, then
    \[
    u^{\Delta}(K)=
\frac{1}{4} \sum_{j=1}^{(n-1)/2} \sum_{i=1}^{j} \alpha_{2i-1}\alpha_{2j} + \frac{1}{8} \{ ( \sum_{i=1}^{(n+1)/2} \alpha_{2i-1} )^{2} -1 \} \quad \big(= a_2(K)\big).
\]
\end{enumerate}
\end{restatable}

In this way, we can present examples of two-bridge knots whose $\Delta$-unknotting numbers are equal to the second coefficients of their Conway polynomials. In addition, we introduce an unknotting technique for two-bridge knots (Claim \ref{claim:technique}).

Furthermore, if we discover a special method to unknot a two-bridge knot, it may allow us to determine its $\Delta$-unknotting number.
Based on this idea, we propose the following conjecture concerning two-bridge knots with $u^{\Delta}(K) = 1$ (Conjecture \ref{thm:conj}). We also provide a partial proof of this conjecture (Theorems \ref{thm:deltaone} and \ref{thm:deltaten}).

\begin{restatable}{conjecture}{conj} \label{thm:conj}
Let $K$ be a nontrivial two-bridge knot. Then, the following two conditions are equivalent. 
\begin{enumerate}
    \item       
    $u^{\Delta}(K) = 1$.
    \item       
    $K$ or $K^*$(the mirror image of $K$) can be expressed as\\ $C(\beta, \beta_n, ... , \beta_2, \beta_1, \overbrace{1, 1, 1, 1, 1}^{\text{5}}, 1-\beta_1, -\beta_2, ... , -\beta_n)$.
\end{enumerate}
\end{restatable}

Theorem \ref{thm:deltaone} establishes the implication (2) $\Rightarrow$ (1) in Conjecture \ref{thm:conj}.

\begin{restatable}{theorem}{deltaone}
\label{thm:deltaone}
Let $K=C(\beta, \beta_n, ... , \beta_2, \beta_1, 1, 1, 1, 1, 1, 1-\beta_1, -\beta_2, ... , -\beta_n)$ be a nontrivial two-bridge knot.
Then, we have $u^{\Delta}(K)= 1$.
\end{restatable}

Theorem \ref{thm:deltaten} is a partial case of Conjecture \ref{thm:conj}.

\begin{restatable}{theorem}{deltaten}
\label{thm:deltaten}
Let $K$ be a two-bridge knot with $u^{\Delta}(K)=1$ and at most 10 crossings, except for $10_{30}$, $10_{36}$, and $10_{38}$. 
Then, $K$ or $K^*$ can be expressed as $C(\beta, \beta_n, ... , \beta_2, \beta_1, 1, 1, 1, 1, 1, 1-\beta_1, -\beta_2, ... , -\beta_n)$.
\end{restatable}

Concerning the ordinary unknotting number for nontrivial two-bridge knots, the following three conditions are equivalent \cite{KaM}.
\begin{enumerate}
    \item       
    $u(K) = 1$.
    \item       
    There exist an odd integer $p(>1)$ and coprime, positive integers $m$ and $n$ with $2mn = p \pm1$ and $K$ is equivalent to $S(p, 2n^2)$.
    \item 
    $K$ can be expressed as $C(\beta, \beta_1, \beta_2, ..., \beta_k, \pm2, -\beta_k, ..., -\beta_2, -\beta_1)$.
\end{enumerate}

As a byproduct, we consider the $\Delta$-Gordian distances of two-bridge knots of type $C(2\beta_1, 2\beta_2)$ (Theorem \ref{thm:deltadelta} and Table 2).
\begin{restatable}{theorem}{deltadelta} \label{thm:deltadelta}
Let $K=C(\alpha_1,\alpha_2)$ and $K^{'}=C(\alpha_1',\alpha_2')$ be two-bridge knots, where $\alpha_1, \alpha_2, \alpha_1', \alpha_2'$ are positive even integers, and $\alpha_1 \geq \alpha_1'$, $\alpha_2 \geq \alpha_2'$.
Then, we have
\begin{equation*}
d_G^{\Delta}(K,K^{'}) = \frac{1}{4} |\alpha_1 \alpha_2 - \alpha_1' \alpha_2'| 
\quad \big(= |a_2(K) - a_2(K^{'})| 
= |u^{\Delta}(K) - u^{\Delta}(K^{'})| \big).
\end{equation*}
\end{restatable}

\vspace{1em}
\section{Preliminaries}\label{sec2}

A knot $K$ is said to be a two-bridge knot if $K$ has a diagram as in Figure \ref{fig:twobridge}, called Conway's normal form. For a knot diagram as in Figure \ref{fig:twobridge}, each $|\alpha_i|$ presents the number of half-twists for integers $\alpha_1, \alpha_2, ... , \alpha_n$. In this paper, for the sign of $\alpha_i$, we assume that a right-handed half-twist is positive if $i$ is odd, and a left-handed half-twist is positive if $i$ is even. We denote this knot diagram by $C(\alpha_1, \alpha_2, ... , \alpha_n)$.
\begin{figure}[htbp]
 \centering
 \includegraphics[width=1\linewidth]{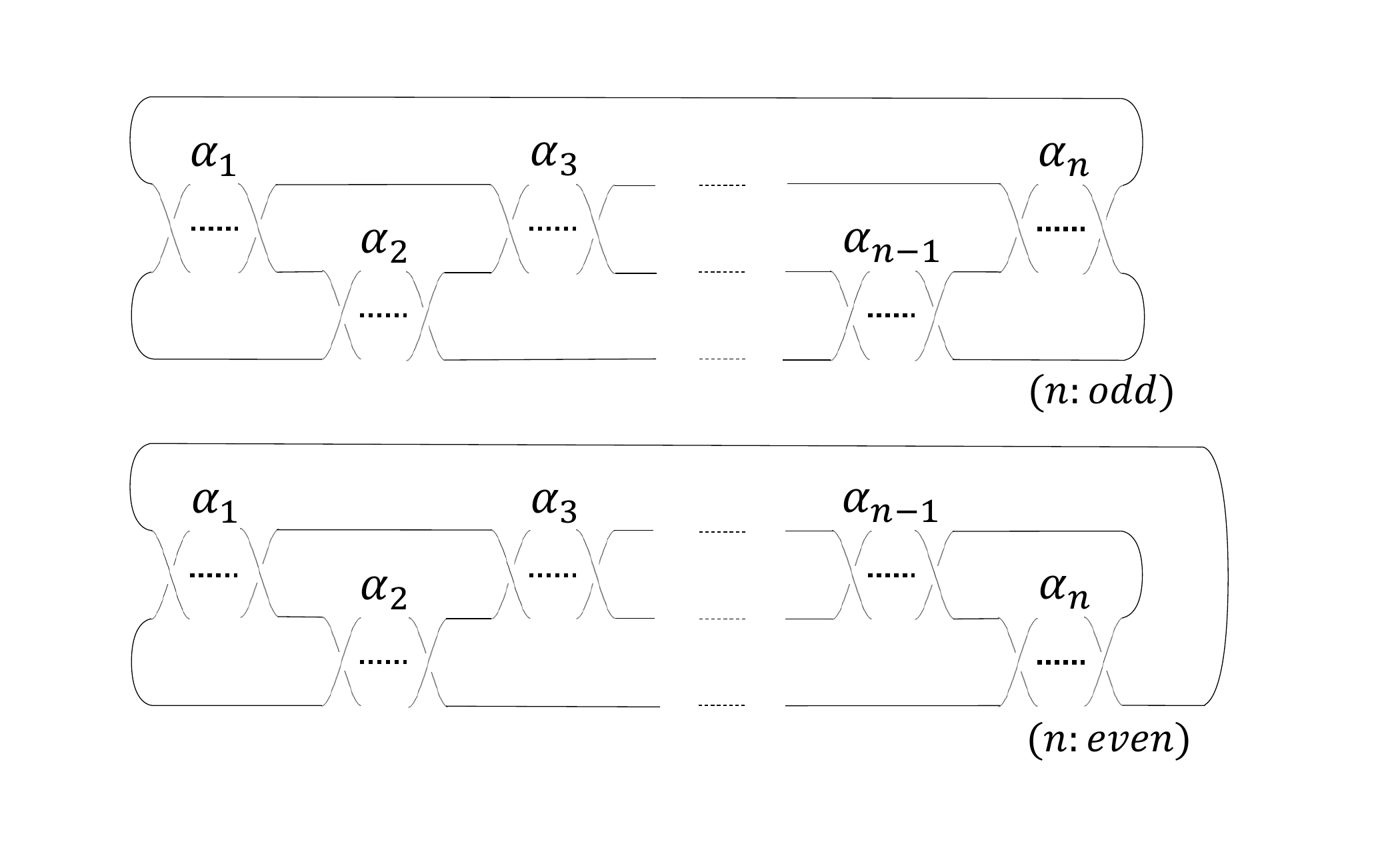}
 \caption{}
 \label{fig:twobridge}
\end{figure}

We use the technique in \cite{MaN, NaNaU}.

\begin{claim}[\cite{MaN}]\label{claim:clasp}
A clasp can leap over a hurdle by a single $\Delta$-move.
\end{claim}

The meaning is given in Figure \ref{fig:clasp}. Here, the clasp is vertically standing on the 2-sphere on which the other part of the knot diagram lies (except for crossings). We use Claim \ref{claim:clasp} to show Claim \ref{claim:change}.
\begin{figure}[htbp]
 \centering
 \includegraphics[width=0.6\linewidth]{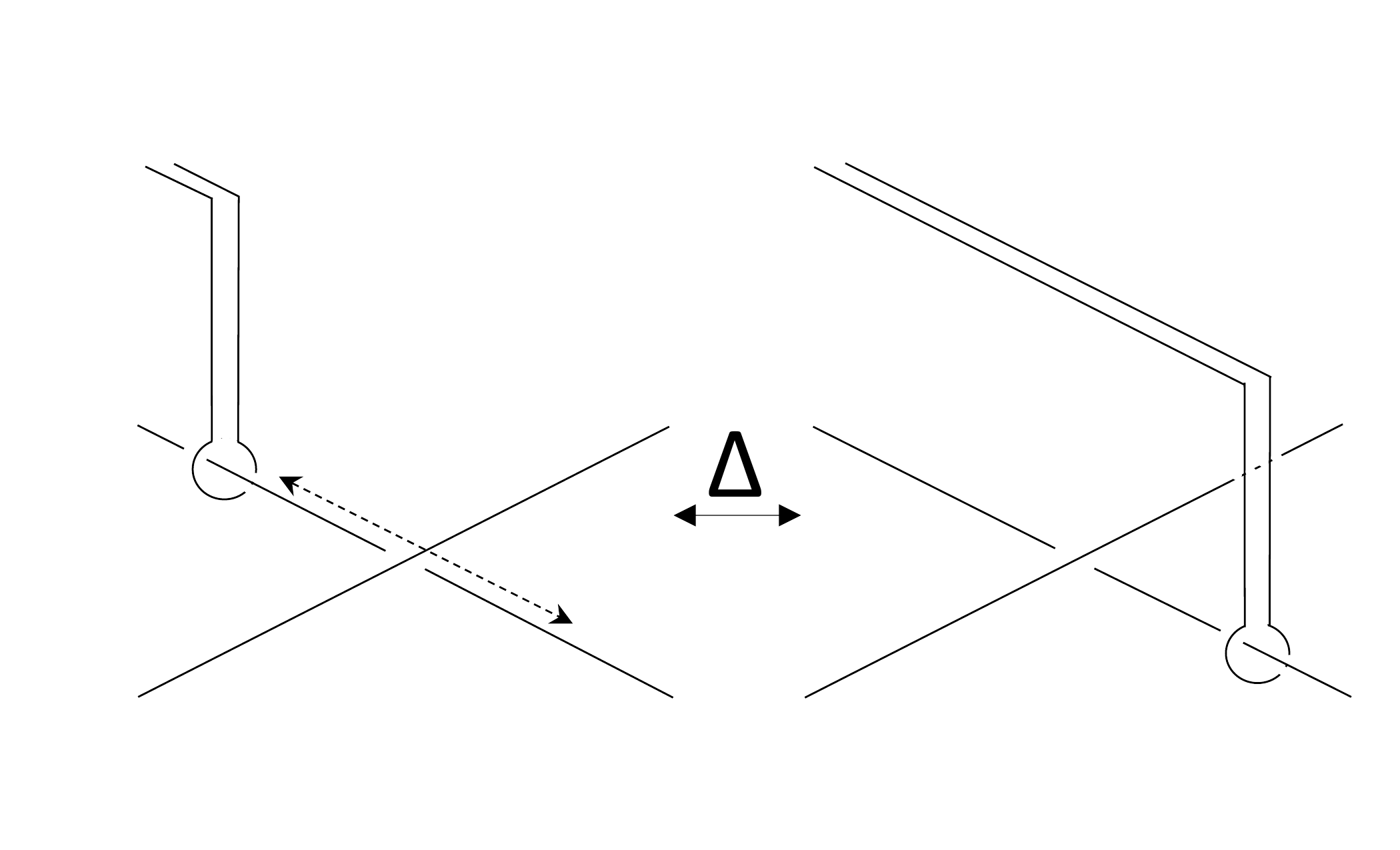}
 \caption{}
 \label{fig:clasp}
\end{figure}

\begin{claim}[\cite{NaNaU}]\label{claim:change}
We can exchange a crossing of a knot by a finite number of $\Delta$-moves.
\end{claim}

The meaning and the proof are given in Figure \ref{fig:tdelta}. Here, the clasp leaps over $t$ hurdles by $t$ times $\Delta$-moves. Then, the crossing labeled with the asterisk * is deformed into the opposite crossing.
\begin{figure}[htbp]
 \centering
 \includegraphics[width=0.8\linewidth]{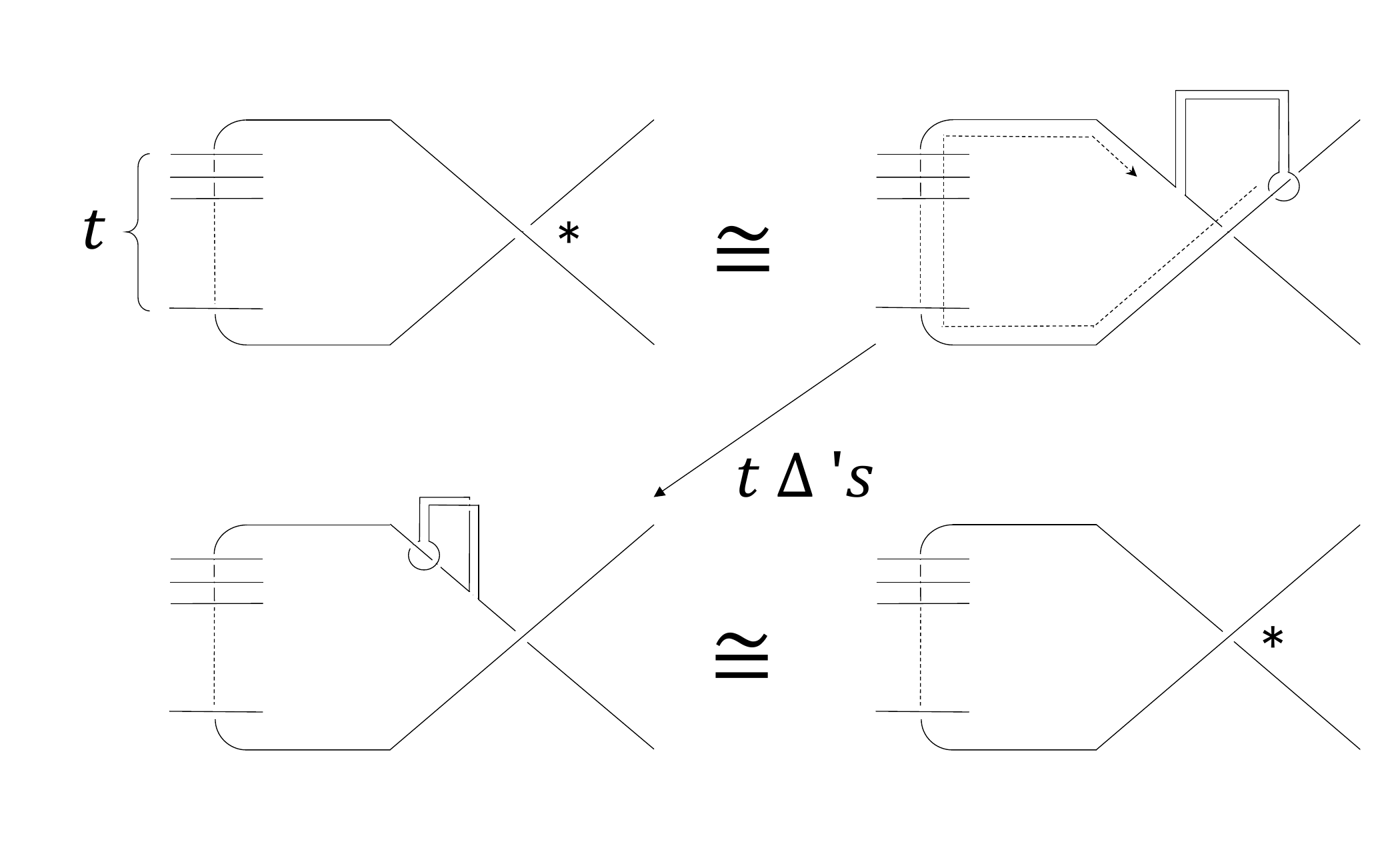}
 \caption{}
 \label{fig:tdelta}
\end{figure}

In particular, Claim \ref{claim:technique} holds for a class of two-bridge knots.
\begin{claim}\label{claim:technique}
Let $K=C(\alpha_1, \alpha_2, ... , \alpha_n)$ be a two-bridge knot, where $\alpha_1$ is a positive even integer and $\alpha_2$ is a positive integer.
In Figure \ref{fig:cct}, if we perform $\frac{1}{2} \alpha_1$ times $\Delta$-moves, we can exchange the crossing labeled with the asterisk *.
Then, we have 
$d_G^{\Delta}(C(\alpha_1, \underline{\alpha_2}, \alpha_3, ... , \alpha_n), C(\alpha_1, \underline{\alpha_2-2}, \alpha_3, ... , \alpha_n)) \leq \frac{1}{2} \alpha_1$.
\end{claim}
\begin{figure}[htbp]
 \centering
 \includegraphics[width=1\linewidth]{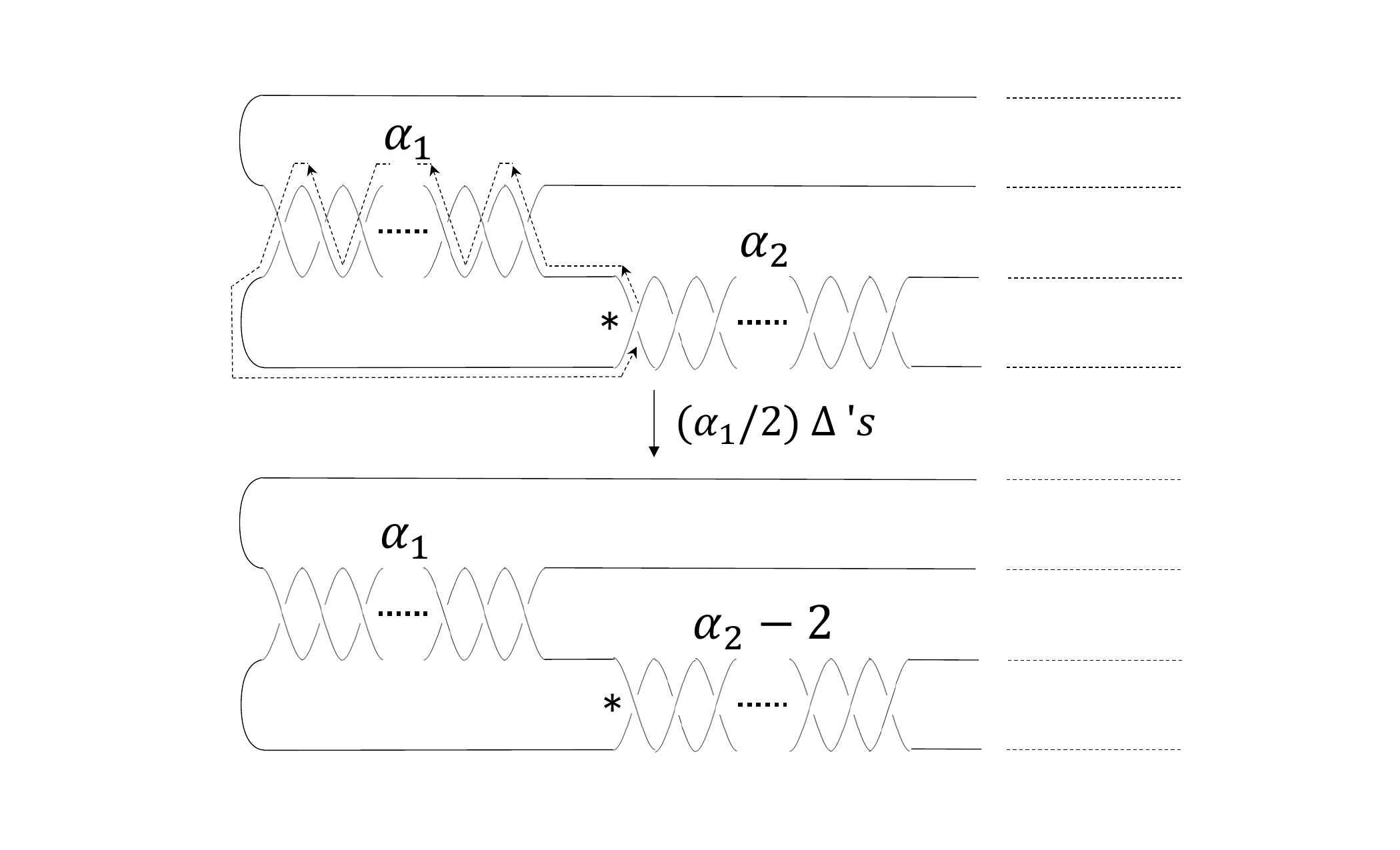}
 \caption{}
 \label{fig:cct}
\end{figure}

\vspace{1em}
\vspace{1em}
\vspace{1em}
In this paper, we use Propositions \ref{prop:lk}, \ref{prop:a2}, and \ref{prop:torus}.

\begin{proposition}\label{prop:lk}
Let $k_+$ be a knot, $k_-$ be a knot obtained from $k_+$ by exchanging a positive crossing into a negative crossing, and $k_0$ a $2$-component link obtained from $k_+$ by smoothing at that crossing. Then $a_2(k_+)-a_2(k_-) = lk(k_0)$.
\end{proposition}

The proof can be found in [2, Chap. III].

\begin{proposition}[\cite{Oka1}]\label{prop:a2}
For any two knots $K$ and $K^{'}$, the difference $d_G^{\Delta}(K,K^{'}) - |a_2(K) - a_2(K^{'})|$ is a non-negative even integer. 
In particular, the difference $u^{\Delta}(K) -|a_2(K)|$ is also a non-negative even integer.
\end{proposition}

\begin{proposition}[\cite{NaNaU}]\label{prop:torus}
Let $T(p,q)$ be the torus knot of type $(p,q)$ for a pair of positive integers $p,q$ with $\gcd(p,q)=1$.
Then we have 
\[
u^{\Delta}(T(p,q)) = \frac{1}{24} (p^2-1)(q^2-1) \quad \big( = a_2(T(p,q)) \big).
\]
\end{proposition}

By deforming a two-bridge knot into $T(2,q)$, we can apply the formula
\[
u^{\Delta}(T(2,q)) = \frac{1}{8} (q^2-1).
\]

\vspace{1em}
\section{Two-bridge knots with $u^{\Delta}(K)=|a_2(K)|$}\label{sec3}

We begin by considering the two-bridge knot $K = C(\alpha_1)$, where $\alpha_1$ is a positive odd integer. 
It is well known that $C(\alpha_1) \cong T(2, \alpha_1)$, and hence we have $u^{\Delta}(K) = |a_2(K)|$.

Next, we investigate the two-bridge knot $K = C(\alpha_1, \alpha_2)$ with $\alpha_1, \alpha_2 > 0$.
Our goal in this section is to show that $u^{\Delta}(K) = |a_2(K)|$.

To this end, we first determine the $\Delta$-unknotting numbers for two-bridge knots of type 
$C(2\beta_1, 2\beta_2, ... ,  2\beta_n)$ and type $C(2\beta_1, 2\beta_2, ... , 2\beta_{n-1}, 2\beta_n-1)$, where $\beta_i$ is a positive integer for $1 \leq i \leq n$.

\subsection{Type $C(2\beta_1, 2\beta_2, ... ,  2\beta_n)$}\label{subsec31}

\even*

\begin{proof}
Let $k_+ = C(\alpha_1, \alpha_2, \alpha_3, ... , \alpha_n)$ and $k_- = C(\alpha_1, \alpha_2-2, \alpha_3, ... , \alpha_n)$. By Proposition \ref{prop:lk}, $a_2(C(\alpha_1, \alpha_2, \alpha_3, ... , \alpha_n)) - a_2(C(\alpha_1, \alpha_2-2, \alpha_3, ... , \alpha_n)) = lk(k_0) = - \frac{1}{2} \alpha_1$.

By repeating the same computation $\frac{1}{2}\alpha_2$ times, we have
\begin{align*}
a_2(C(\alpha_1, \underline{\alpha_2}, \alpha_3, ... , \alpha_n)) & - a_2(C(\alpha_1, \underline{\alpha_2-2}, \alpha_3, ..., \alpha_n)) = - \frac{1}{2} \alpha_1 \quad, ... , \\
a_2(C(\alpha_1, \underline{2}, \alpha_3, ... , \alpha_n)) & - a_2(C(\alpha_1, \underline{0}, \alpha_3, ... , \alpha_n)) = - \frac{1}{2} \alpha_1 .
\end{align*}

Here, $C(\alpha_1, 0, \alpha_3, \alpha_4, ... , \alpha_n) \cong C(\alpha_1 + \alpha_3, \alpha_4, ... , \alpha_n)$.

By continuing with another $\frac{1}{2} \alpha_4$ steps, we have
\begin{align*}
a_2(C(\alpha_1 + \alpha_3, \underline{\alpha_4}, \alpha_5, ... , \alpha_n)) & - a_2(C(\alpha_1 + \alpha_3, \underline{\alpha_4-2}, \alpha_5, ... , \alpha_n)) \\
& = - \frac{1}{2} ( \alpha_1 + \alpha_3 )  \quad, ... , \\
a_2(C(\alpha_1 + \alpha_3, \underline{2}, \alpha_5, ... , \alpha_n)) & - a_2(C(\alpha_1 + \alpha_3, \underline{0}, \alpha_5, ... , \alpha_n)) \\
& = - \frac{1}{2} ( \alpha_1 + \alpha_3 ) .
\end{align*}

Here, $C(\alpha_1 + \alpha_3, 0, \alpha_5, \alpha_6, ... , \alpha_n) \cong C(\alpha_1 + \alpha_3 + \alpha_5, \alpha_6, ... , \alpha_n)$.

Proceeding further, we have 
\begin{align*}
a_2(C(\alpha_1 + \alpha_3 + ... + \alpha_{n-1}, \underline{\alpha_n})) & - a_2(C(\alpha_1 + \alpha_3 + ... + \alpha_{n-1}, \underline{\alpha_n-2})) \\
& = - \frac{1}{2} ( \alpha_1 + \alpha_3 + ... + \alpha_{n-1} )
\quad, ... ,  \\
a_2(C(\alpha_1 + \alpha_3 + ... + \alpha_{n-1}, \underline{2})) & - a_2(C(\alpha_1 + \alpha_3 + ... + \alpha_{n-1}, \underline{0})) \\
& = - \frac{1}{2} ( \alpha_1 + \alpha_3 + ... + \alpha_{n-1} ) .
\end{align*}

Here, $C(\alpha_1 + \alpha_3 + ... + \alpha_{n-1}, 0) \cong O$, hence $a_2(C(\alpha_1 + \alpha_3 + ... + \alpha_{n-1}, 0) ) = 0$.

By summing up, we obtain 
\begin{align*}
& a_2(C(\alpha_1, \alpha_2, ... , \alpha_n)) \\
& = - \frac{1}{4} \{ \alpha_1 \alpha_2 + (\alpha_1 + \alpha_3)\alpha_4 + ... + (\alpha_1 + \alpha_3 + ... + \alpha_{n-1}) \alpha_n \} \\
& = - \frac{1}{4} \sum_{j=1}^{n/2} \sum_{i=1}^{j} \alpha_{2i-1}\alpha_{2j} .
\end{align*}

On the other hand, if we perform $\frac{1}{2} \alpha_1$ times $\Delta$-moves (Claim \ref{claim:technique}), we have $d_G^{\Delta}(C(\alpha_1, \alpha_2, \alpha_3, ... , \alpha_n), C(\alpha_1, \alpha_2-2, \alpha_3, ... , \alpha_n)) \leq \frac{1}{2} \alpha_1$. 

By repeating the same computation $\frac{1}{2} \alpha_2$ times, we have
\begin{gather*}
d_G^{\Delta}(C(\alpha_1, \underline{\alpha_2}, \alpha_3, ... , \alpha_n), C(\alpha_1, \underline{\alpha_2-2}, \alpha_3, ... , \alpha_n)) \leq \frac{1}{2} \alpha_1 \quad, ... , \\
d_G^{\Delta}(C(\alpha_1, \underline{2}, \alpha_3, ... , \alpha_n), C(\alpha_1, \underline{0}, \alpha_3, ... , \alpha_n)) \leq \frac{1}{2} \alpha_1 . 
\end{gather*}

By continuing with another $\frac{1}{2} \alpha_4$ steps, we have
\begin{gather*}
d_G^{\Delta}(C(\alpha_1 + \alpha_3, \underline{\alpha_4}, \alpha_5, ... , \alpha_n), C(\alpha_1 + \alpha_3, \underline{\alpha_4-2}, \alpha_5, ... , \alpha_n)) \leq \frac{1}{2} ( \alpha_1 + \alpha_3 ) \quad, ... , \\
d_G^{\Delta}(C(\alpha_1 + \alpha_3, \underline{2}, \alpha_5, ... , \alpha_n), C(\alpha_1 + \alpha_3, \underline{0}, \alpha_5, ... , \alpha_n)) \leq \frac{1}{2} ( \alpha_1 + \alpha_3 ). 
\end{gather*}

Proceeding further, we have 
\begin{gather*}
d_G^{\Delta}(C(\alpha_1 + \alpha_3 + ... + \alpha_{n-1}, \underline{\alpha_n}), C(\alpha_1 + \alpha_3 + ... + \alpha_{n-1}, \underline{\alpha_n-2}) ) \\
\leq \frac{1}{2} ( \alpha_1 + \alpha_3 + ... + \alpha_{n-1} ) \quad, ... , \\
d_G^{\Delta}(C(\alpha_1 + \alpha_3 + ... + \alpha_{n-1}, \underline{2}), C(\alpha_1 + \alpha_3 + ... + \alpha_{n-1}, \underline{0}))  \\
\leq \frac{1}{2} ( \alpha_1 + \alpha_3 + ... + \alpha_{n-1} ) .
\end{gather*}

By summing up, we obtain 
\begin{align*}
& d_G^{\Delta}(C(\alpha_1, \alpha_2, ... , \alpha_n), O) \\
& \leq d_G^{\Delta}(C(\alpha_1, \alpha_2, ... , \alpha_n), C(\alpha_1, \alpha_2-2, ... , \alpha_n)) + ... \\
& + d_G^{\Delta}(C(\alpha_1 + \alpha_3 + ... + \alpha_{n-1}, 2), C(\alpha_1 + \alpha_3 + ... + \alpha_{n-1}, 0)) \\
& \leq \frac{1}{4} \{ \alpha_1 \alpha_2 + (\alpha_1 + \alpha_3)\alpha_4 + ... + (\alpha_1 + \alpha_3 +... + \alpha_{n-1})\alpha_n \} \\
& = -a_2(K) = |a_2(K)|.
\end{align*}

Therefore, we have $u^{\Delta}(K) \leq |a_2(K)|$.

By Proposition \ref{prop:a2}, we also have $u^{\Delta}(K) \geq |a_2(K)|$. 
Thus, we conclude that $u^{\Delta}(K)=|a_2(K)|$.
The proof is complete.
\end{proof}

\subsection{Type $C(2\beta_1, 2\beta_2, ... , 2\beta_{n-1}, 2\beta_n-1)$}\label{subsec32}

\odd*

\begin{proof}
Proceeding in the same way as in the proof of Theorem \ref{thm:even}, we obtain Theorem \ref{thm:odd}.

(1) If $n$ is even, then we have
\begin{align*}
a_2(C(\alpha_1, \alpha_2, \alpha_3, ... , \alpha_n)) & - a_2(C(\alpha_1, \alpha_2-2, \alpha_3, ..., \alpha_n)) =  \frac{1}{2} \alpha_1 \quad, ... , \\
a_2(C(\alpha_1 + \alpha_3 + ... + \alpha_{n-1}, \underline {3})) & - a_2(C(\alpha_1 + \alpha_3 + ... + \alpha_{n-1}, \underline {1})) \\
& = \frac{1}{2} ( \alpha_1 + \alpha_3 + ... + \alpha_{n-1} ).
\end{align*}

Here, $C(\alpha_1 + \alpha_3 + ... + \alpha_{n-1} , 1) \cong C(\alpha_1 + \alpha_3 + ... + \alpha_{n-1} + 1) \cong T(2, \alpha_1 + \alpha_3 + ... + \alpha_{n-1} + 1)$, 
hence $a_2(C(\alpha_1 + \alpha_3 + ... + \alpha_{n-1}, 1) ) = \frac{1}{8} \{ ( \alpha_1 + \alpha_3 + ... + \alpha_{n-1} + 1 )^2 -1 \}$.

By summing up, we obtain 
\begin{align*}
& a_2(C(\alpha_1, \alpha_2, ... , \alpha_n)) \\
& = \frac{1}{4} \{ \alpha_1 \alpha_2 + (\alpha_1 + \alpha_3)\alpha_4 + ... + ( \alpha_1 + \alpha_3 + ... + \alpha_{n-3} ) \alpha_{n-2} \\
& + ( \alpha_1  +  \alpha_3 + ... + \alpha_{n-1} ) (\alpha_n-1) \} 
+ \frac{1}{8} \{ ( \alpha_1 + \alpha_3 + ... + \alpha_{n-1} + 1 )^2 -1 \}\\
& = \frac{1}{4} \sum_{j=1}^{n/2} \sum_{i=1}^{j} \alpha_{2i-1}\alpha_{2j} 
+ \frac{1}{8} ( \sum_{i=1}^{n/2} \alpha_{2i-1} )^{2}.
\end{align*}

On the other hand, we perform $\Delta$-moves (Claim \ref{claim:technique}), we have
\begin{gather*}
d_G^{\Delta}(C(\alpha_1, \alpha_2, \alpha_3, ... , \alpha_n), C(\alpha_1, \alpha_2-2, \alpha_3, ... , \alpha_n)) \leq \frac{1}{2} \alpha_1 \quad, ... , \\
d_G^{\Delta}(C(\alpha_1 + \alpha_3 + ... + \alpha_{n-1}, \underline{3}), C(\alpha_1 + \alpha_3 + ... + \alpha_{n-1}, \underline{1}))  \\
\leq \frac{1}{2} ( \alpha_1 + \alpha_3 + ... + \alpha_{n-1} ) .
\end{gather*}

By summing up, we obtain 
\begin{align*}
& d_G^{\Delta}(C(\alpha_1, \alpha_2, ... , \alpha_n), O) \\
& \leq d_G^{\Delta}(C(\alpha_1, \alpha_2, ... , \alpha_n), C(\alpha_1, \alpha_2-2, ... , \alpha_n)) + ... \\
& + d_G^{\Delta}(C(\alpha_1 + \alpha_3 + ... + \alpha_{n-1}, 3), C(\alpha_1 + \alpha_3 + ... + \alpha_{n-1}, 1)) \\
& + d_G^{\Delta}(C(\alpha_1 + \alpha_3 + ... + \alpha_{n-1}, 1), O) \\
& \leq \frac{1}{4} \{ \alpha_1 \alpha_2 + (\alpha_1 + \alpha_3)\alpha_4 + ... + ( \alpha_1 + \alpha_3 + ... + \alpha_{n-3} ) \alpha_{n-2} \\
& + ( \alpha_1  +  \alpha_3 + ... + \alpha_{n-1} ) (\alpha_n-1) \} 
+ \frac{1}{8} \{ (\alpha_1 + \alpha_3 + ... + \alpha_{n-1} + 1)^2 -1 \}\\
& = \frac{1}{4} \sum_{j=1}^{n/2} \sum_{i=1}^{j} \alpha_{2i-1}\alpha_{2j} + \frac{1}{8} ( \sum_{i=1}^{n/2} \alpha_{2i-1} )^{2} = a_2(K) = |a_2(K)|.
\end{align*}

Therefore, we have $u^{\Delta}(K) \leq |a_2(K)|$.

By Proposition \ref{prop:a2}, we also have $u^{\Delta}(K) \geq |a_2(K)|$. 
Thus, we conclude that $u^{\Delta}(K)=|a_2(K)|$. \\

(2) If $n$ =1 ($K=C(\alpha_1)$), then $\alpha_1$ is a positive odd integer. Here, $K=C(\alpha_1) \cong T(2, \alpha_1)$. Then, we have $u^{\Delta}(K) = \frac{1}{8} (\alpha_1^2 - 1)$.

If $n$ is odd and $n \geq 3$, then we have
\begin{align*}
a_2(C(\alpha_1, \alpha_2, \alpha_3, ... , \alpha_n)) & - a_2(C(\alpha_1, \alpha_2-2, \alpha_3, ..., \alpha_n)) =  \frac{1}{2} \alpha_1 \quad, ... , \\
a_2(C(\alpha_1 + \alpha_3 + ... + \alpha_{n-2}, \underline {2}, \alpha_n)) & - a_2(C(\alpha_1 + \alpha_3 + ... + \alpha_{n-2}, \underline {0}, \alpha_n)) \\
& = \frac{1}{2} ( \alpha_1 + \alpha_3 + ... + \alpha_{n-2} ).
\end{align*}

Here, $C(\alpha_1 + \alpha_3 + ... + \alpha_{n-2} , 0, \alpha_n) \cong C(\alpha_1 + \alpha_3 + ... + \alpha_{n-2} + \alpha_n) \cong  T(2, \alpha_1 + \alpha_3 + ... + \alpha_{n-2} + \alpha_n)$,
hence $a_2(C(\alpha_1 + \alpha_3 + ... + \alpha_{n-2}, 0, \alpha_n )) = \frac{1}{8} \{ (\alpha_1 + \alpha_3 + ... + \alpha_{n-2} + \alpha_n )^2 -1 \}$.

By summing up, we obtain 
\begin{align*}
& a_2(C(\alpha_1, \alpha_2, ... , \alpha_n)) \\
& = \frac{1}{4} \{ \alpha_1 \alpha_2 + (\alpha_1 + \alpha_3)\alpha_4 + ... + ( \alpha_1 + \alpha_3 + ... + \alpha_{n-2} ) \alpha_{n-1} \} \\
& + \frac{1}{8} \{ ( \alpha_1 + \alpha_3 + ... + \alpha_n )^2 -1 \}\\
& = \frac{1}{4} \sum_{j=1}^{(n-1)/2} \sum_{i=1}^{j} \alpha_{2i-1}\alpha_{2j} + \frac{1}{8} \{ ( \sum_{i=1}^{(n+1)/2} \alpha_{2i-1} )^{2} -1 \} .
\end{align*}

On the other hand, we perform $\Delta$-moves (Claim \ref{claim:technique}), we have
\begin{gather*}
d_G^{\Delta}(C(\alpha_1, \alpha_2, \alpha_3, ... , \alpha_n), C(\alpha_1, \alpha_2-2, \alpha_3, ... , \alpha_n)) \leq \frac{1}{2} \alpha_1 \quad, ... , \\
d_G^{\Delta}(C(\alpha_1 + \alpha_3 + ... + \alpha_{n-2}, 2, \alpha_n), C(\alpha_1 + \alpha_3 + ... + \alpha_{n-2}, 0, \alpha_n))  \\
\leq \frac{1}{2} ( \alpha_1 + \alpha_3 + ... + \alpha_{n-2} ) .
\end{gather*}

By summing up, we obtain 
\begin{align*}
& d_G^{\Delta}(C(\alpha_1, \alpha_2, ... , \alpha_n), O) \\
& \leq d_G^{\Delta}(C(\alpha_1, \alpha_2, ... , \alpha_n), C(\alpha_1, \alpha_2-2, ... , \alpha_n)) + ... \\
& + d_G^{\Delta}(C(\alpha_1 + \alpha_3 + ... + \alpha_{n-2}, 2, \alpha_n), C(\alpha_1 + \alpha_3 + ... + \alpha_{n-2}, 0, \alpha_n)) \\
& + d_G^{\Delta}(C(\alpha_1 + \alpha_3 + ... + \alpha_{n-2}, 0, \alpha_n), O) \\
& \leq \frac{1}{4} \{ \alpha_1 \alpha_2 + (\alpha_1 + \alpha_3)\alpha_4 + ... + ( \alpha_1  +  \alpha_3 + ... + \alpha_{n-2} ) \alpha_{n-1} \}  \\
& + \frac{1}{8} \{ (\alpha_1 + \alpha_3 + ... + \alpha_{n-2} + \alpha_n)^2 -1 \}\\
& = \frac{1}{4} \sum_{j=1}^{(n-1)/2} \sum_{i=1}^{j} \alpha_{2i-1}\alpha_{2j} 
+ \frac{1}{8} \{ ( \sum_{i=1}^{(n+1)/2} \alpha_{2i-1} )^{2} -1 \}  = a_2(K) = |a_2(K)|.
\end{align*}

Therefore, we have $u^{\Delta}(K) \leq |a_2(K)|$.

By Proposition \ref{prop:a2}, we also have $u^{\Delta}(K) \geq |a_2(K)|$. 
Thus, we conclude that $u^{\Delta}(K)=|a_2(K)|$.

The proof is complete.
\end{proof}

\subsection{Type $C(\alpha_1, \alpha_2)$}\label{subsec33}

\begin{corollary}\label{cor:a1a2}
Let $K=C(\alpha_1, \alpha_2)$ be a two-bridge knot, where $\alpha_1 > 0$ and $\alpha_2 > 0$.
Then, we have $u^{\Delta}(K)=|a_2(K)|$.
\begin{enumerate}
    \item
    If both $\alpha_1$ and $\alpha_2$ are even integers, then
    
    \[
    u^{\Delta}(K) = \frac{1}{4}\alpha_1 \alpha_2 \quad \big(= -a_2(K)\big).
    \]
    \item
    If $\alpha_1$ is an even integer and $\alpha_2$ is an odd integer, then
    
    \[
    u^{\Delta}(K) = \frac{1}{8}\alpha_1(\alpha_1 + 2\alpha_2) \quad \big(= a_2(K)\big).
    \]
\end{enumerate}
\end{corollary}

\begin{proof}
By Theorem \ref{thm:even}, if $n=2$, then $u^{\Delta}(K) = \frac{1}{4}\alpha_1 \alpha_2 \quad \big(= -a_2(K)\big)$. 

By Theorem \ref{thm:odd}(1), if $n=2$, then
$u^{\Delta}(K) = \frac{1}{8}\alpha_1(\alpha_1 + 2\alpha_2) \quad \big(= a_2(K)\big)$.

The proof is complete.
\end{proof}

\vspace{1em}
\section{Two-bridge knots with $u^{\Delta}(K) \neq |a_2(K)|$}\label{sec4}

Let $K=C(\alpha_1, \alpha_2, \alpha_3)$ be a two-bridge knot with at most 10 crossings. Among such knots, only $8_2, 10_2$, and $10_6$ satisfy the inequality $u^{\Delta}(K)\neq|a_2(K)|$.
Specifically, $8_2 \cong C(5, 1, 2)$, $10_2 \cong C(7, 1, 2)$, and $10_6 \cong C(5, 3, 2)$. 

In this section, we investigate two-bridge knots with $u^{\Delta}(K) \neq |a_2(K)|$.

\subsection{Type $C(2\beta_1, 2\beta_2-1, 2\beta_3-1)$}\label{subsec41}

\begin{example}\label{ex:a1a2a3}
Let $K=C(\alpha_1, \alpha_2, \alpha_3)$ be a two-bridge knot, where $\alpha_1$ is a positive even integer, and both $\alpha_2$ and $\alpha_3$ are positive odd integers. Define $v^{\Delta}(K)=\frac{1}{4} \alpha_1(\alpha_2 + \alpha_3) + 
\frac{1}{8} (\alpha_3^2 -8\alpha_3 + 7)$.
Then, we have
\begin{align*}
& a_2(K) = -\frac{1}{4} \alpha_1(\alpha_2 + \alpha_3) 
+ \frac{1}{8} (\alpha_3^2 -1), \\
& u^{\Delta}(K) \leq v^{\Delta}(K) 
\quad \big(\leq  \frac{1}{4} \alpha_1(\alpha_2 + \alpha_3) + \frac{1}{8} (\alpha_3^2 -1) \big).
\end{align*}
\begin{enumerate}
    \item
    If $\alpha_3 = 1$ or $3$, then $|a_2(K)| = v^{\Delta}(K)$. Then, we have $u^{\Delta}(K) = |a_2(K)|$.

    \item
     If $\alpha_3 \geq 5$, then $|a_2(K)| < v^{\Delta}(K)$.

Let $K=C(2, 1, \alpha_3)$ or $C(2, \alpha_2, 5)$. Then, we have $v^{\Delta}(K) = |a_2(K)| + 2$, so that $ u^{\Delta}(K) = |a_2(K)|$ or $|a_2(K)| + 2$.

In fact, we have the following examples:
\begin{align*}
u^{\Delta}(C(2,1,5)) = |a_2(C(2,1,5))| + 2, \\
u^{\Delta}(C(2,1,7)) = |a_2(C(2,1,7))| + 2, \\
u^{\Delta}(C(2,3,5)) = |a_2(C(2,3,5))| + 2.
\end{align*}
\end{enumerate}
\end{example}

\vspace{1em}
Proceeding in the same way as in the proof of Theorem \ref{thm:even}, we have 
\begin{align*}
a_2(C(\alpha_1, \alpha_2, \alpha_3)) & - a_2(C(\alpha_1 - 2, \alpha_2, \alpha_3)) = -\frac{1}{2} (\alpha_2 + \alpha_3) \quad, ... , \\
a_2(C(2, \alpha_2, \alpha_3)) & - a_2(C(0, \alpha_2, \alpha_3)) = -\frac{1}{2} (\alpha_2 + \alpha_3) .
\end{align*}

Here, $C(0, \alpha_2, \alpha_3) \cong C(\alpha_3)$.

Proceeding further, we have 
\begin{align*}
a_2(C(\alpha_3)) & - a_2(C(\alpha_3 - 2)) = \frac{1}{2} (\alpha_3 - 1) \quad, ... , \\
a_2(C(5)) & - a_2(C(3)) = \frac{1}{2} (5 - 1) , \\
a_2(C(3)) & - a_2(C(1)) = \frac{1}{2} (3 - 1) .
\end{align*}

Here, $C(1) \cong O$, hence $a_2(C(1)) = 0$.

By summing up, we obtain 
\begin{align*}
a_2(C(\alpha_1, \alpha_2, \alpha_3)) & = -\frac{1}{4} \alpha_1(\alpha_2 + \alpha_3) + \{ \frac{1}{2} (\alpha_3 - 1) + ... + 2 +1 \} \\
& = -\frac{1}{4} \alpha_1(\alpha_2 + \alpha_3) + 
\frac{1}{8} (\alpha_3^2 -1).
\end{align*}

On the other hand, we perform $\Delta$-moves (Claim \ref{claim:technique}), we have
\begin{gather*}
d_G^{\Delta}(C(\alpha_1, \alpha_2, \alpha_3), C(\alpha_1, \alpha_2-2, \alpha_3)) \leq \frac{1}{2} \alpha_1 \quad, ... , \\
d_G^{\Delta}(C(\alpha_1, 1, \alpha_3), C(\alpha_1, -1, \alpha_3)) \leq \frac{1}{2} \alpha_1.
\end{gather*}

Here, $C(\alpha_1, -1, \alpha_3) \cong C(\alpha_1 - 1, 1 - \alpha_3)$.

Proceeding further, we have 
\begin{gather*}
d_G^{\Delta}(C(\alpha_1 - 1, 1 - \alpha_3), C(\alpha_1 - 3, 1 - \alpha_3)) \leq \frac{1}{2} (\alpha_3 -1) \quad, ... , \\
d_G^{\Delta}(C(3, 1 - \alpha_3), C(1, 1 - \alpha_3 )) \leq \frac{1}{2} (\alpha_3 - 1).
\end{gather*}

Here, $C(1, 1 - \alpha_3) \cong T(2, \alpha_3 -2)$, hence $d_G^{\Delta}(C(1, 1 - \alpha_3), O) = \frac{1}{8} \{ (\alpha_3 - 2)^2 - 1\}$.

By summing up, we obtain 
\begin{align*}
& d_G^{\Delta}(C(\alpha_1, \alpha_2, \alpha_3), O) \\
& \leq d_G^{\Delta}(C(\alpha_1, \alpha_2, \alpha_3), C(\alpha_1, \alpha_2 - 2, \alpha_3)) + ... \\
& + d_G^{\Delta}(C(\alpha_1, 1, \alpha_3), C(\alpha_1, -1, \alpha_3)) \\
& + d_G^{\Delta}(C(\alpha_1 - 1, 1 - \alpha_3), C(\alpha_1 - 3, 1 - \alpha_3))  + ... \\
& + d_G^{\Delta}(C(3, 1 - \alpha_3)), C(1, 1 -\alpha_3)) \\
& + d_G^{\Delta}(C(1, 1 - \alpha_3)), O) \\
& \leq \frac{1}{4} \{ \alpha_1 ( \alpha_2 + 1 ) + (\alpha_3 - 1)(\alpha_1 - 2) \}  \\
& + \frac{1}{8} \{ (\alpha_3 - 2)^2 - 1\} \\
& =\frac{1}{4} \alpha_1(\alpha_2 + \alpha_3) + 
\frac{1}{8} (\alpha_3^2 -8\alpha_3 + 7) = v^{\Delta}(K).
\end{align*}

\subsection{Type $C(3, m, 2, 1, 2)$}\label{subsec42}

\begin{example}\label{ex:3m212}
Let $K=C(3, m, 2, 1, 2)$ be a two-bridge knot, where $m$ is a positive even integer.
Then, we have $u^{\Delta}(K) = 2$ and $a_2(K) = 0$, so that $u^{\Delta}(K) \ne |a_2(K)|$.
\end{example}

Proceeding in the same way as in the proof of Theorem \ref{thm:even}, we have 
\begin{align*}
a_2(C(3, m, 2, 1, 2)) & - a_2(C(3, m-2, 2, 1, 2)) = 0 \quad, ... , \\
a_2(C(3, 4, 2, 1, 2)) & - a_2(C(3, 2, 2, 1, 2)) = 0.
\end{align*}

By summing up, we obtain 
\begin{align*}
a_2(C(3, m, 2, 1, 2) )- a_2(C(3, 2, 2, 1, 2)) = 0.
\end{align*}

Here, $C(3, 2, 2, 1, 2) \cong 10_{25}$, hence $a_2(C(3, 2, 2, 1, 2)) = 0$.

Thus, we have $a_2(K) = 0$.

On the other hand, as illustrated in Figure \ref{fig:3m212}, $K$ can be deformed into $3_1$ by a $\Delta$-move. 
Therefore, we have $u^{\Delta}(K) = d_G^{\Delta}(K, O) \leq d_G^{\Delta}(K, 3_1) + d_G^{\Delta}(3_1, O) = 1 + 1 = 2$. 
By Proposition \ref{prop:a2}, we conclude that $u^{\Delta}(K) = 2$.
\begin{figure}[htbp]
 \centering
 \includegraphics[width=1\linewidth]{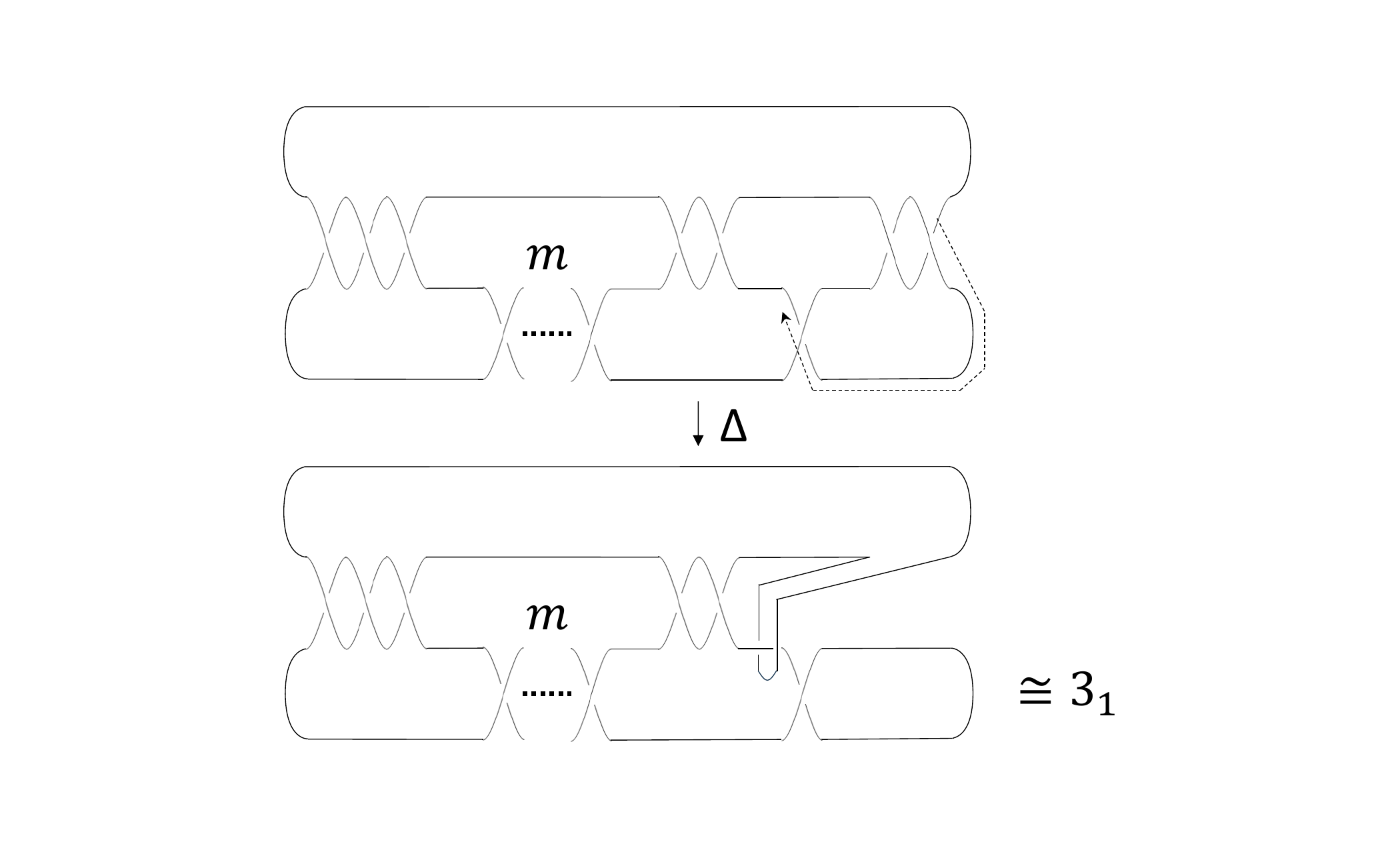}
 \caption{}
 \label{fig:3m212}
\end{figure}

\vspace{1em}
\section{Two-bridge knots with $u^{\Delta}(K)=1$}\label{sec5}

In this section, we investigate two-bridge knots with $u^{\Delta}(K) = 1$.

At the Topology Symposium in 1996, we considered two-bridge knots with $u^{\Delta}(K)=1$. In a slightly different form, we now propose the following conjecture.

\conj*
\begin{figure}[htbp]
 \centering
 \includegraphics[width=1\linewidth]{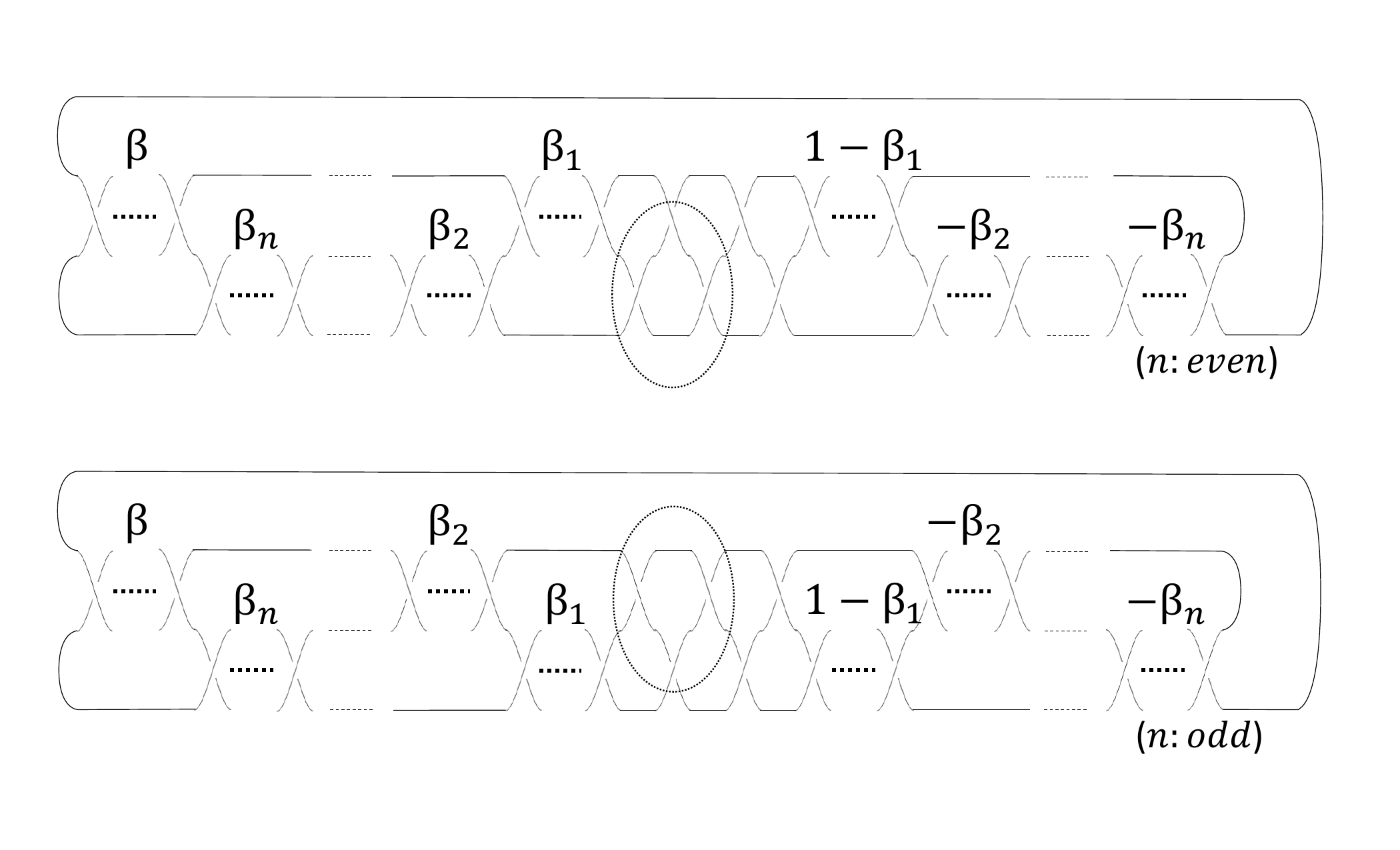}
 \caption{}
 \label{fig:deltaone}
\end{figure}

\deltaone*

\begin{proof}
If we perform a $\Delta$-move on the dotted circle in Figure \ref{fig:deltaone}, $K$ can be deformed into the trivial knot $O$. The proof is complete.
\end{proof}

\deltaten*

\begin{proof}
See Table 1. 
In this table, we present some examples, though other representations are also possible. The proof is complete.
\end{proof}

\begin{table}
\begin{center}
\caption{}
\begin{tabular}{|l|l|c|}
\hline
$K$ & Conway's normal form & Examples 
\\ \hline\hline
$3_1$ & $C(3)$ & 
$C(-2,0,\overbrace{1,1,1,1,1},1)$,
$C(3,-1,\overbrace{1,1,1,1,1},2)$
\\ \hline
$4_1$ & $C(2,2)$ & $C(-1, 0, \overbrace{1,1,1,1,1},1)$,
$C(2,-1,\overbrace{1,1,1,1,1},2)$
\\ \hline
$6_2$ & $C(3,1,2)$ & $C(-3,0,\overbrace{1,1,1,1,1},1)$,
$C(4,-1,\overbrace{1,1,1,1,1},2)$
\\ \hline
$6_3$ & $C(2,1,1,2)$ & $C(0,0,\overbrace{1,1,1,1,1},1)$,
$C(1,-1,\overbrace{1,1,1,1,1},2)$ 
\\ \hline
$7_6$ & $C(2,2,1,2)$ & $C(-4,0,\overbrace{1,1,1,1,1},1)$, 
$C(5,-1,\overbrace{1,1,1,1,1},2)$
\\ \hline
$7_7$ & $C(2,1,1,1,2)$ & $C(1,0,\overbrace{1,1,1,1,1},1)$,
$C(0,-1,\overbrace{1,1,1,1,1},2)$
\\ \hline
$8_{11}$ & $C(3,2,1,2)$ & $C(-5,0,\overbrace{1,1,1,1,1},1)$, 
$C(6,-1,\overbrace{1,1,1,1,1},2)$
\\ \hline
$8_{13}$ & $C(3,1,1,1,2)$ & $C(2,0,\overbrace{1,1,1,1,1},1)$, 
$C(-1,-1,\overbrace{1,1,1,1,1},2)$
\\ \hline
$9_{12}$ & $C(4,2,1,2)$ & $C(-6,0,\overbrace{1,1,1,1,1},1)$, 
$C(7,-1,\overbrace{1,1,1,1,1},2)$
\\ \hline
$9_{14}$ & $C(4,1,1,1,2)$ & $C(3,0,\overbrace{1,1,1,1,1},1)$, 
$C(-2,-1,\overbrace{1,1,1,1,1},2)$
\\ \hline
$10_7$ & $C(5,2,1,2)$ & $C(-7,0,\overbrace{1,1,1,1,1},1)$, 
$C(8,-1,\overbrace{1,1,1,1,1},2)$
\\ \hline
$10_{10}$ & $C(5,1,1,1,2)$ & $C(4,0,\overbrace{1,1,1,1,1},1)$, 
$C(-3,-1,\overbrace{1,1,1,1,1},2)$
\\ \hline
$10_{19}$ & $C(4,1,1,1,3)$ & $C(1,3,\overbrace{1,1,1,1,1},-2)$,
$C(2,-2,\overbrace{1,1,1,1,1},3)$
\\ \hline
$10_{32}$ & $C(3,1,1,1,2,2)$ & $C(-2,3,\overbrace{1,1,1,1,1},-2)$,
$C(-1,-2, \overbrace{1,1,1,1,1},3)$
\\ \hline
\end{tabular}
\end{center}
\end{table}

\vspace{1em}
\section{$\Delta$-Gordian distance for two-bridge knots}\label{sec6}

In this section, we investigate the $\Delta$-Gordian distances of the two-bridge knots previously discussed.

\subsection{Type $C(2\beta_1, 2\beta_2)$}\label{subsec61}

\deltadelta*

\begin{proof}
By Corollary \ref{cor:a1a2}(1), then $a_2(C(\alpha_1, \alpha_2)) = - \frac{1}{4} \alpha_1 \alpha_2$ and $a_2(C(\alpha_1', \alpha_2')) = - \frac{1}{4} \alpha_1' \alpha_2'$.

On the other hand, if we perform $\frac{1}{2} \alpha_1$ times $\Delta$-moves (Claim \ref{claim:technique}), we have $d_G^{\Delta}(C(\alpha_1, \alpha_2), C(\alpha_1, \alpha_2-2)) \leq \frac{1}{2} \alpha_1$. 

By repeating the same computation $\frac{1}{2} (\alpha_2 - \alpha_2')$ times, we have
\begin{gather*}
d_G^{\Delta}(C(\alpha_1, \alpha_2), C(\alpha_1, \alpha_2-2)) \leq \frac{1}{2} \alpha_1 \quad, ... , \\
d_G^{\Delta}(C(\alpha_1, \alpha_2'+2), C(\alpha_1, \alpha_2')) \leq \frac{1}{2} \alpha_1 . 
\end{gather*}

By summing up, we obtain 
\begin{align*}
&d_G^{\Delta}(C(\alpha_1, \alpha_2), C(\alpha_1, \alpha_2')) \\
& \leq d_G^{\Delta}(C(\alpha_1, \alpha_2), C(\alpha_1, \alpha_2-2)) + ... + d_G^{\Delta}(C(\alpha_1, \alpha_2'+2), C(\alpha_1, \alpha_2'))  \\
& \leq \frac{1}{4} \alpha_1(\alpha_2 - \alpha_2' ).
\end{align*}

By the same argument, we have 
\begin{gather*}
d_G^{\Delta}(C(\alpha_1, \alpha_2'), C(\alpha_1', \alpha_2')) \leq \frac{1}{4} \alpha_2'(\alpha_1 - \alpha_1' ) .
\end{gather*}

Therefore, we have
\begin{align*}
d_G^{\Delta}(C(\alpha_1, \alpha_2), C(\alpha_1', \alpha_2')) & \leq d_G^{\Delta}(C(\alpha_1, \alpha_2), C(\alpha_1, \alpha_2')) + d_G^{\Delta}(C(\alpha_1, \alpha_2'), C(\alpha_1', \alpha_2')) \\ & \leq \frac{1}{4} \alpha_1( \alpha_2 - \alpha_2' ) +  \frac{1}{4} \alpha_2' ( \alpha_1 - \alpha_1' ) = \frac{1}{4} ( \alpha_1 \alpha_2 - \alpha_1' \alpha_2' )
\\ & = |a_2(C(\alpha_1, \alpha_2)) - a_2(C(\alpha_1', \alpha_2'))|.
\end{align*}

By Proposition \ref{prop:a2}, we also have 
\begin{align*}
d_G^{\Delta}(C(\alpha_1, \alpha_2), C(\alpha_1', \alpha_2')) \geq |a_2(C(\alpha_1, \alpha_2)) - a_2(C(\alpha_1', \alpha_2'))|. 
\end{align*}

Thus, we conclude that 
\begin{align*}
d_G^{\Delta}(C(\alpha_1, \alpha_2), C(\alpha_1', \alpha_2')) & = |a_2(C(\alpha_1, \alpha_2)) - a_2(C(\alpha_1', \alpha_2'))| \\ 
& = \frac{1}{4} |\alpha_1 \alpha_2 - \alpha_1' \alpha_2'| = |u^{\Delta}(K) - u^{\Delta}(K^{'})|. 
\end{align*}

The proof is complete.
\end{proof}

In this case, as illustrated below, we can observe multiple shortest deformation paths from $C(\alpha_1,\alpha_2)$ to $C(\alpha_1',\alpha_2')$.
\begin{align*}
& C(\alpha_1,\alpha_2) &\overset{(\alpha_2/2){\Delta}'s}{\to} 
& C(\alpha_1-2,\alpha_2) &\overset{(\alpha_2/2){\Delta}'s}{\to} 
& ... &\overset{(\alpha_2/2){\Delta}'s}{\to}
& C(\alpha_1',\alpha_2) \\
& \downarrow {(\alpha_1/2){\Delta}'s}
&& \downarrow {\{(\alpha_1-2)/2\}{\Delta}'s}
&&
&& \downarrow {(\alpha_1'/2){\Delta}'s}
\\
& C(\alpha_1,\alpha_2-2) &  \overset{}{\to} \quad
& C(\alpha_1-2,\alpha_2-2) & \overset{}{\to} \quad 
& ... & \overset{}{\to} \quad 
& C(\alpha_1',\alpha_2-2) \\
& \downarrow {(\alpha_1/2){\Delta}'s}
&& \downarrow {\{(\alpha_1-2)/2\}{\Delta}'s}
&&
&& \downarrow {(\alpha_1'/2){\Delta}'s}
\\
& \quad \vdots 
&& \quad \vdots
&&& \quad \quad \quad \quad \vdots
\\
& \downarrow {(\alpha_1/2){\Delta}'s}
&& \downarrow {\{(\alpha_1-2)/2\}{\Delta}'s}
&&
&& \downarrow {(\alpha_1'/2){\Delta}'s}
\\
& C(\alpha_1,\alpha_2') &\overset{(\alpha_2'/2){\Delta}'s}{\to}
& C(\alpha_1-2,\alpha_2') &\overset{(\alpha_2'/2){\Delta}'s}{\to}
& ... &\overset{(\alpha_2'/2){\Delta}'s}{\to}
& C(\alpha_1',\alpha_2') 
\end{align*}

By Theorem \ref{thm:deltadelta}, we obtain Table 2. In this table, we do not consider $\Delta$-Gordian distances between the mirror images of these knots.
\begin{table}[b]
\begin{center}
\caption{}
\begin{tabular}{|ll||l|l|l|l|l|l|}
\hline
$K$ & & $4_1$ & $6_1$ & $8_1$ & $8_3$ & $10_1$ & $10_3$ \\
& & $(2,2)$ & $C(4,2)$ & $C(6,2)$ & $C(4,4)$ & $C(8,2)$ & $C(6,4)$ \\
\hline \hline
$4_1$ & $C(2,2)$ & $0$ & $1$ & $2$ & $3$ & $3$ & $5$  \\
\hline
$6_1$ & $C(4,2)$ && $0$ & $1$ & $2$ & $2$ &  $4$ \\ \hline
$8_1$ & $C(6,2)$ &&& $0$ & $1$ or $3$ & $1$ &  $3$ \\ \hline
$8_3$ & $C(4,4)$ &&&& $0$ & $2$ or $4$ &  $2$ \\ \hline
$10_1$ & $C(8,2)$ &&&&& $0$ &  $2$ or $4$ \\ \hline
$10_3$ & $C(6,4)$ &&&&&&  $0$ \\ \hline
\end{tabular}
\end{center}
\end{table}

In this case, as illustrated below, we can observe two shortest deformation paths from $C(6,4)$ to $C(4,2)$.
\begin{align*}
&&& C(6,4) \cong 10_3 &\overset{2{\Delta}'s}{\to} \quad \quad
& C(4,4) \cong 8_3 &&
\\
&&& \downarrow {3{\Delta}'s}
&& \downarrow {2{\Delta}'s} &&
\\
& C(8,2) \cong 10_1 & \overset{1{\Delta}}{\to} \quad \quad
& C(6,2) \cong 8_1 & \overset{1{\Delta}}{\to} \quad \quad
& C(4,2) \cong 6_1 & \overset{1{\Delta}}{\to} \quad \quad
& C(2,2) \cong 4_1
\end{align*}

\subsection{Type $C(2, m, 5)$}\label{subsec62}

\begin{example}\label{}
Let $K=C(2, m, 5)$ and $K^{'}=C(2, m', 5)$ be two-bridge knots, where $m$ and $m'$ are positive odd integers.
Then, we have
$d_G^{\Delta}(K,K^{'}) = \frac{1}{2} |m - m'| 
\quad \big(= |a_2(K) - a_2(K^{'})|\big)$. 
\end{example}

See Example \ref{ex:a1a2a3}. Then, we have $a_2(K) = \frac{1}{2}(1-m)$ and $d_G^{\Delta}(C(2, m, 5), C(2, m-2, 5))=1$. Therefore, we have $|a_2(K) -a_2(K^{'})| = \frac{1}{2} |m - m'|$ and $d_G^{\Delta}(K,K^{'}) \leq \frac{1}{2} |m - m'|$.
By Proposition \ref{prop:a2}, we conclude that $d_G^{\Delta}(K,K^{'}) = \frac{1}{2} |m - m'|$.

\subsection{Type $C(3, m, 2, 1, 2)$}\label{subsec63}

\begin{example}\label{}
Let $K=C(3, m, 2, 1, 2)$ and $K^{'}=C(3, m', 2, 1, 2)$ be two-bridge knots, where $m$ and $m'$ are positive even integers and $K \not\cong K^{'}$.
Then, we have $d_G^{\Delta}(K,K^{'})= 2$.
\end{example}

In this case,
$|u^{\Delta}(K) - u^{\Delta}(K^{'})| < d_G^{\Delta}(K,K^{'}) < u^{\Delta}(K) + u^{\Delta}(K^{'})$.

See Example \ref{ex:3m212} and Figure \ref{fig:3m212}. $K$ can be deformed into $3_1$ by a $\Delta$-move. 
Therefore, we have $d_G^{\Delta}(K,K^{'}) \leq d_G^{\Delta}(K,3_1) + d_G^{\Delta}(3_1,K^{'}) = 1 + 1 = 2$. 
By Proposition \ref{prop:a2}, we conclude that $d_G^{\Delta}(K,K^{'}) = 2$.

\subsection{Type $C(\beta, \beta_n, ... , \beta_2, \beta_1, 1, 1, 1, 1, 1, 1-\beta_1, -\beta_2, ... , -\beta_n)$}\label{subsec64}

\begin{example}\label{}
Let $K=C(\beta, \beta_n, ... , \beta_2, \beta_1, 1, 1, 1, 1, 1, 1-\beta_1, -\beta_2, ... , -\beta_n)$ and $K^{'}=C(\beta', \beta_n', ... , \beta_2', \beta_1', 1, 1, 1, 1, 1, 1-\beta_1', -\beta_2', ... , -\beta_n')$ be nontrivial two-bridge knots, where $K \not\cong K^{'}$.
Then, we have
$d_G^{\Delta}(K,K^{'})= 2 
\quad \big(= u^{\Delta}(K) + u^{\Delta}(K^{'})\big)$.
\end{example}

By Theorem \ref{thm:deltaone}, we have $u^{\Delta}(K) = 1$ and $u^{\Delta}(K^{'}) = 1$.
Therefore, we have $d_G^{\Delta}(K,K^{'}) \leq d_G^{\Delta}(K,O) + d_G^{\Delta}(O,K^{'}) = u^{\Delta}(K) + u^{\Delta}(K^{'}) = 1 + 1 = 2$.
By Proposition \ref{prop:a2}, we conclude that $d_G^{\Delta}(K,K^{'})= 2$.

\section*{Acknowledgments}
The author would like to thank Professor Makoto Sakuma for his valuable advice and continuous support.

He also thanks Professors Yasutaka Nakanishi, Hitoshi Murakami, and Yoshiaki Uchida for their helpful comments and suggestions.

Finally, he would like to thank his family for their support and understanding throughout the course of this work.




\end{document}